\documentclass[10pt,reqno]{amsart}
\usepackage[cp1251]{inputenc}
\usepackage[english]{babel}
\usepackage{amsmath}
\usepackage{amssymb}
\usepackage{amsfonts}
\usepackage{amsaddr}

\usepackage{cite}
\usepackage[dvips]{graphicx}

\newtheorem{theorem}{Theorem}

\newtheorem{corollary}{Corollary}
\newtheorem{proposition}{Proposition}

\begin{document}
\thispagestyle{empty}
\sloppy

\title[On Binomial coefficients of real arguments]
{On Binomial coefficients\\of real arguments}

\let\MakeUppercase\relax 
\author{\rm\Large Tatiana~I.~Fedoryaeva}
\address{\it Sobolev Institute of Mathematics, Novosibirsk, Russia\\
E-mail{\rm:} fti@math.nsc.ru }
\thanks{\sc Fedoryaeva,~T.I., On binomial coefficients of real arguments}
\thanks{\rm The work was carried out within the framework of the state contract of
the Sobolev Institute of Mathematics (project no. FWNF-2022-0018)}
\thanks{\it Submitted to arXiv June, 7, 2022}

\vspace{1cm}
\maketitle
{\small
\begin{quote}
\noindent{\sc\bf Abstract.} As is well-known, a generalization of
the classical concept of the factorial $n!$ for a real number $x\in
{\mathbb R}$ is the value of Euler's gamma function $\Gamma(1+x)$.
In this connection, the notion of a binomial coefficient naturally
arose for admissible values of the real arguments.

By elementary means, it is proved  a number of properties of
binomial coefficients $\binom{r}{\alpha}$ of real arguments
$r,\,\alpha\in {\mathbb R}$\, such as analogs of unimodality,
symmetry, Pascal's triangle, etc. for classical binomial
coefficients. The asymptotic behavior of such generalized binomial
coefficients of a special form is established.

\medskip


\medskip

\end{quote}
}
\vspace{1cm}

\section*{Introduction}
\hspace*{\parindent} We study binomial coefficients of real
arguments. The aim of the investigation is to obtain by elementary
methods analogs of the basic properties well known for the classic
binomial coefficients. Such properties are of independent interest
and, in addition, can simplify the work with binomial coefficients
of the form $\binom{n}{m}$ with integer non-negative arguments $n$
and $m$, given essentially by real values with considered rounding
to an integer (when, for example, floor and ceiling functions for a
real number are used, etc.). So, for example, the properties of
unimodality and symmetry allow passing from such binomial
coefficients $\binom{n}{m}$, $0\leq m \leq n$ to \/"close"\/ real
binomial coefficients of the form $\binom{r}{\alpha}$, $\alpha\in
(-1,r+1)$ and vice versa. This approach simplifies the evaluation of
expressions with discrete binomial coefficients with integer
arguments of the specified form.

Note that the binomial coefficients of the form $\binom{r}{n}$,
where $r\in {\mathbb R}$ and $n\in {\mathbb N}$, can be defined in
the standard way as
$$\binom{r}{n}=\frac{r(r-1)(r-2)\cdots (r-n+1)}{n!}.$$
\noindent This approach was discussed in [4], where a numerous
number of identities for such binomial coefficients is given.  In
[3], D.~Fowler studied the graph of the function $\binom{r}{\alpha}$
of two real variables $r$ and $\alpha$, various slices of this graph
were constructed using a computer and their analysis was carried
out. It is also indicated there an explicit expression for the
binomial coefficient $\binom{n}{\alpha}$, where $n$ is a
non-negative integer, through elementary functions (see\/
Proposition $2$ in Section $2$). On the basis of this
representation, Stuart~T.~Smith investigated the binomial
coefficients of the form $\binom{n}{z}$ with complex variable $z\in
{\mathbb C}$ and fixed natural number $n\in {\mathbb N}$, a number
of properties of such a function of complex argument $z$ is
established in [6]. In particular, the derivatives of the first and
second orders are calculated, and for the real argument $z$,
increasing and decreasing intervals, zeros of the function, etc are
found. It is also noted there the nontriviality of the function
investigation $\binom{n}{\alpha}$ of real variable $\alpha$ exactly
on the interval $\alpha\in (-1,n+1)$, in contrast to the domain
outside this interval. In particular, the increasing and decreasing
of this function was established rather difficult.

In this paper, we prove by elementary means a number of properties
of the binomial coefficients $\binom{r}{\alpha}$ of real arguments
$r,\,\alpha\in {\mathbb R}$, $\alpha\in (-1,r+ 1)$ (analogs of the
properties of unimodality, symmetry, Pascal's triangle, etc. for
discrete binomial coefficients), which may be useful in further
research (see, for example, [1]).

\section{Preliminary information}
\hspace*{\parindent} The article uses the generally accepted
concepts and notation of real analysis [2], as well as the standard
concepts of combinatorial analysis [4]. Denote by $(a,b)$ the {\it
open real interval}\/ between the numbers $a,b\in {\mathbb R}$,
$o(1)$ is an {\it infinitesimal function} in a neighborhood of
$\infty$, $n!$ is the {\it factorial} of non-negative integer $n$,
i.e. $n!=n(n-1)\cdots 2\cdot 1$, and wherein we define $0!=1$,
$\binom{n}{m}$, where $0\leq m \leq n$, is the (standard) binomial
coefficient (with non-negative integer arguments $n$, $m$ ), i.e.
$$\binom{n}{m}=\frac{n!}{m!(n-m)!}.$$
To denote {\it asymptotic equality} of real-valued functions $f(x)$
and $g(x)$ as $x\rightarrow\infty$, we use the notation $f(x)\sim
g(x)$, which by definition means that $f(x)=g(x)(1+r(x))$ in some
neighborhood of $\infty$, where $r(x)=o(1)$, or, equivalently (for
functions positive in some neighborhood of $\infty$)
$$\lim_{x\rightarrow\infty}\frac{f(x)}{g(x)}=1.$$

The standard approach is considered, according to which the concept
of factorial for non-negative integers extends to real (and even
complex) numbers by the {\it gamma function} $\Gamma(\alpha)$. We
will use its definition in the following {\it Euler-Gauss form}
[2, p. 393--394, 812]

\begin{equation}
\Gamma(\alpha)=\lim_{n\rightarrow\infty} \frac{(n-1)!\,n^
{\alpha}}{\alpha(\alpha +1)(\alpha +2)\cdots(\alpha
+n-1)},\,\alpha\in{\mathbb R}\setminus \{0,-1,-2,\ldots\},
\end{equation}
such a limit exists for any specified value of $\alpha$ (see, for
example, [5] or [2, p. 393]). In view of the problem statement, we
do not consider extensions of the gamma function outside its
standard domain of definition. Note that when defining the gamma
function in the form of the {\it Euler integral of the second kind}
$$\Gamma(\alpha)=\int_0^{\infty} x^{\alpha-1} e^{-x} dx,$$
converging for $\alpha >0$, we obtain an equivalent definition on
the interval $(0,\infty)$ [2, p. 811]. The gamma function
$\Gamma(\alpha)$ is continuous and has continuous derivatives of all
orders on $(0,\infty)$, has no real roots, and it is positive on
$(0,\infty)$. For any non-negative integer $n$ the following
equality holds
\begin{eqnarray}
\Gamma(1+n)=n!\,,
\end{eqnarray}
moreover, the next {\it reduction formula} is valid [2, p. 394]
\begin{eqnarray}
\Gamma(1+\alpha)=\alpha\Gamma(\alpha),\, \alpha\in {\mathbb
R}\setminus \{0,-1,-2,\ldots\}.
\end{eqnarray}
In addition, the series expansion of $\sin \pi \alpha$ and (1) imply
the following {\it reflection formula} (see, for example, [2, 5])
\begin{eqnarray}
\Gamma(\alpha)\,\Gamma(1-\alpha)=\frac{\pi}{\sin\pi\alpha}\,,\,
\alpha\in {\mathbb R}\setminus \{0,\pm1,\pm2,\ldots\}.
\end{eqnarray}

As a generalization of the discrete binomial coefficient
$\binom{n}{m}$, the binomial coefficient of real arguments is
defined as follows (see, for example, [3])
\begin{eqnarray}
\binom{r}{\alpha}=\frac{\Gamma(1+r)}{\Gamma(1+\alpha)\Gamma(1+r-\alpha)}\,.
\end{eqnarray}
Note that if $r\in(-1,+\infty)$ and $\alpha\in (-1,r+1)$, the
binomial coefficient $\binom{r}{\alpha}$ is defined correctly by the
equality (5).

\section{Binomial coefficients $\binom{r}{\alpha}$ for $r,\,\alpha\in {\mathbb R}$}
\begin{theorem}[properties of the binomial coefficient of real arguments]
Let  $r\in(-1,+\infty)$ and $\alpha\in (-1,r+1)$. Then \\[-6pt]

{\rm (i)} $\binom{r}{\alpha}>0$, $\binom{r}{0}=1$ and $\binom{r}{r}=1;$\\[-3pt]

{\rm (ii)} $\binom{0}{\alpha}=\left\{
\begin{array}{ll}
\,1, & \mbox{if } \alpha=0,\\
\frac{\sin\pi\alpha}{\pi\alpha}\,, & \mbox{if } \alpha\neq 0 \mbox{
and } \alpha\in (-1,1);
\end{array} \right.$\\[1pt]

{\rm (iii)} $\binom{r}{r-\alpha}=\binom{r}{\alpha};$\\[-3pt]

{\rm (iv)} $\binom{r}{\alpha}=\binom{r-1}{\alpha -1} + \binom{r-1}{\alpha},$
if $r\in(0,+\infty)$ and $\alpha\in (0,r);$\\[-3pt]

{\rm (v)} binomial coefficient $\phi(\alpha)\!=\binom{r}{\alpha}$ is
strictly increasing on the interval \,$(-1,\frac{r}{2}]$ and
strictly decreasing on the interval \,$[\frac{r}{2},r+1);$\\[-3pt]

{\rm (vi)} binomial coefficient $\psi(r)\!=\binom{r}{\alpha}$ is
strictly increasing for $\alpha >0$, strictly decreasing for
$-1<\alpha <0$ and $\psi(r)\!\equiv1$ if $\alpha=0$.
\end{theorem}
\begin{proof}
Statement (i) follows from the relations (2), (5).

Prove (ii). If $\alpha =0$, the required equality follows from (i).
Further, we assume that $\alpha\neq 0$. Using the relations
(2)--(5), we obtain
$$\binom{0}{\alpha}=\frac{\Gamma(1)}{\Gamma(1+\alpha)\Gamma(1-\alpha)}=
\frac{1}{\alpha\Gamma(\alpha)\Gamma(1-\alpha)}=\frac{\sin\pi\alpha}{\pi\alpha}\,.$$

Note that if $\alpha\in (-1,r+1)$, then $r-\alpha\in (-1,r+1)$.
Therefore, the binomial coefficient $\binom{r}{r-\alpha}$ is defined
and the required equality from (iii) is satisfied due to (5). It is
also easy to prove (iv) from (3) and (5).

Prove statement (v). Let $\alpha$, $\beta\in (-1,r+1)$. From (1) we
obtain
\begin{equation*}
\Gamma(1+\alpha)\,\Gamma(1+r-\alpha)=\lim_{n\rightarrow\infty}
\frac{(n-1)!\,(n-1)!\,n^{2+r}}{\prod_{i=1}^n \, (\alpha
+i)(r-\alpha+i)}\,.
\end{equation*}
Hence,
\begin{equation}
\frac{\phi(\alpha)}{\phi(\beta)}=
\frac{\Gamma(1+\beta)\,\Gamma(1+r-\beta)}{\Gamma(1+\alpha)\,\Gamma(1+r-\alpha)}=
\lim_{n\rightarrow\infty} \prod_{i=1}^n \,\delta_i(\alpha,\beta),\,
\mbox{ where}
\end{equation}
$$\delta_i(\alpha,\beta)= \frac{(\alpha +i)(r-\alpha+i)}{(\beta
+i)(r-\beta+i)}\,.$$ Note that $\delta_i(\alpha,\beta)>0$ for every
$\alpha$, $\beta\in (-1,r+1)$ and $i=1,\ldots,n$. It is also easy to
prove that
\begin{equation}
\delta_i(\alpha,\beta)\geq 1 \Leftrightarrow f(\alpha)\geq f(\beta),
\end{equation}
where $f(x)=-x^2+xr$ and parabola $f(x)$ is strictly increasing on
$(-\infty,\frac{r}{2}]$ as well as strictly decreasing on $[\frac
{r}{2},+\infty)$. Moreover, it is directly established that

\begin{equation}
\delta_1(\alpha,\beta)=1+\varepsilon(\alpha,\beta), \mbox{ where }
\varepsilon(\alpha,\beta)=\frac{(r-\alpha-\beta)(\alpha-\beta)}{(\beta+1)(r-\beta+1)}.
\end{equation}

Let $-1<\beta <\alpha \leq \frac{r}{2}$. Then $f(\alpha)>f(\beta)$
and $\varepsilon(\alpha,\beta)>0$. By virtue of (7), we have
$\delta_i(\alpha,\beta)\geq 1$ for every $i=1,\ldots,n$. Hence, from
(6) and (8) we obtain
$$\frac{\phi(\alpha)}{\phi(\beta)}\geq
\delta_1(\alpha,\beta)=1+\varepsilon(\alpha,\beta)>1.$$

Similarly, if $\frac{r}{2}\leq \beta <\alpha < r+1$, then
$f(\alpha)<f(\beta)$ and $\varepsilon(\alpha,\beta)<0$. Therefore,
$0<\delta_i(\alpha,\beta)<1$, $i=1,\ldots,n$ and
$$\frac{\phi(\alpha)}{\phi(\beta)}\leq
\delta_1(\alpha,\beta)=1+\varepsilon(\alpha,\beta)<1.$$

Prove statement (vi). In view of statement (i), we can assume that
$\alpha\neq 0$. Let $r<r'$. Note that $1+r+i>0$, $1+r'+i>0$ and
$\alpha/(1+r+i)<1$, $\alpha/(1+r'+i)<1$ for every $i\geq 0$. From
(1) we obtain
$$\frac{\Gamma(1+r)}{\Gamma(1+r-\alpha)}=\Bigl(1-
\frac{\alpha}{1+r}\Bigr)\lim_{n\rightarrow\infty}
n^{\alpha}\prod_{i=1}^{n-1}  \Bigl(1-
\frac{\alpha}{1+r+i}\Bigr)\,.$$ Hence,
$\Gamma(1+r)/\Gamma(1+r-\alpha) < \Gamma(1+r')/\Gamma(1+r'-\alpha)$
for $\alpha >0$ (and the reverse strict inequality holds for $\alpha
<0$). In view of (5), we conclude $\psi(r)<\psi(r')$ (respectively
$\psi(r)> \psi(r')$ for $\alpha <0$).
\end{proof}
\begin{proposition}
Let $r$ takes real values, $\alpha\in {\mathbb R}$ does not depend
on $r$ and $0<\alpha <1$. Then the following asymptotic equality is
valid as $r$ tends to infinity
\begin{equation}
\binom{r}{r\alpha}\sim \sqrt{\frac{1}{2\pi \alpha
(1-\alpha)r}}\,\,\Bigl(\frac{1}{\alpha}\Bigr)^{\alpha
r}\,\Bigl(\frac{1}{1-\alpha}\Bigr)^{(1-\alpha) r}.
\end{equation}
\end{proposition}
\begin{proof}
For $r>0$ the functions $\Gamma(1+r)$, $\Gamma(1+r\alpha)$,
$\Gamma(1+r-r\alpha)$ are defined correctly and positive. By virtue
of (5), we have
\begin{eqnarray}
\binom{r}{r\alpha}=\frac{\Gamma(1+r)}{\Gamma(1+r\alpha)\,\Gamma(1+r-r\alpha)}\,.
\end{eqnarray}
For the gamma function, the following generalized Stirling formula
is valid (see, for example, [2]):
$$\Gamma(1+x)\sim \sqrt{2\pi x}\,\Bigl(\frac{x}{e}\Bigr)^{x}\mbox{
as } x\rightarrow \infty.$$ In view of the condition $0<\alpha <1$,
we have $r\alpha\rightarrow \infty$ and $r-r\alpha\rightarrow
\infty$ as $r$ tends to infinity. Now, using the generalized
Stirling formula, we obtain by equivalent transformations the
asymptotic equality (9) from (10) .
\end{proof}
\begin{corollary}
Let $r$ takes non-negative integer values, $\alpha\in {\mathbb R}$
does not depend on $r$ and $0<\alpha <1$. Then the asymptotic
equality {\rm (9)} is valid as $r$ tends to infinity.
\end{corollary}
\begin{proof}
Let $g(r)$ be the function on the right side of the asymptotic
equality (9) and $f(r)=\binom{r}{r\alpha}/g(r)$. Then
$\lim_{r\rightarrow\infty} f(r)=1$ by Proposition $1$. Non-negative
integers $n$ form an infinitesimal subsequence of the values of real
variable $r$. Therefore, the existing limit value of function $f(r)$
of the real argument as $r\rightarrow \infty$ is preserved for
function $f(n)$ of the non-negative integer argument as
$n\rightarrow \infty$.
\end{proof}

As noted in [3], in the case of binomial coefficients of the form
$\binom{n}{\alpha}$ when $n$ is a non-negative integer, the binomial
coefficient is explicitly expressed in terms of elementary
functions. The following proposition formalizes this statement and
its justification is based on the properties of the gamma function
and binomial coefficients of real arguments.
\renewcommand{\theproposition}{\arabic{proposition} {\rm[3]}}
\begin{proposition}
Let $n$ be a non-negative integer and real number $\alpha\in
(-1,n+1)$. Then the following equality is valid
\[
\binom{n}{\alpha}=\left\{
\begin{array}{ll}
\frac{\sin \pi\alpha}{\pi\alpha}\,, & \mbox{if } n=0 \mbox{ and }
\alpha\neq 0,\\[5pt]
\frac{n!}{(n-\alpha)(n-1-\alpha)\cdots (1-\alpha)}\, \frac{\sin
\pi\alpha}{\pi\alpha}\,, & \mbox{if } n\geq 1 \mbox{ and } \alpha\notin\{0,1,\ldots,n\},\\[5pt]
\frac{n!}{\alpha!(n-\alpha)!}\,, & \mbox{if } \alpha\in
\{0,1,\ldots,n\}.
\end{array} \right.
\]
\end{proposition}
\renewcommand{\theproposition}{\arabic{proposition}}
\begin{proof}
The required equality is proved in Theorem $1$ for $n=0$, and for
$\alpha\in \{0,1,\ldots,n\}$ it follows from the property of the
gamma function (2). Let now $n\geq 1$ and
$\alpha\notin\{0,1,\ldots,n\}$. Suppose that $n-i-\alpha\in
\{0,-1,-2,\ldots\}$ for some $i\in {\mathbb N}$ and $0\leq i \leq
n-1$. Then $\alpha-(n-i)\in \{0,1,2,\ldots\}$. Hence, $\alpha\in
{\mathbb N}$ and therefore $\alpha\in \{0,1,\ldots,n\}$, got a
contradiction. Thus, $n-i-\alpha\notin \{0,-1,-2,\ldots\}$ for every
$i=0,1,\ldots,n-1$. By virtue of the reduction formula (3), we have
$$\Gamma(1+n-\alpha)=(n-\alpha)\Gamma(n-\alpha)=\ldots
=\prod_{i=0}^{n-1}(n-i-\alpha)\Gamma(1-\alpha).$$
Since $\alpha\in {\mathbb R}\setminus \{0,\pm1,\pm2,\ldots\}$, from
(3)--(5) we obtain the required expression for the binomial
coefficient $\binom{n}{\alpha}$.
\end{proof}


\bigskip
\begin{thebibliography}{6}

\bibitem{FTI22-2}
T.I.~Fedoryaeva, {\it Logarithmic asymptotic of the number of
central vertices of almost all $n$-vertex graphs of diameter $k$},
Siber. Electr. Math. Reports, to appear.
\bibitem{FGM}
G.M.~Fikhtengol'ts,  {\it Course of Differential and Integral
Calculus Volume 2},  Fizmatlit, Moscow, 2003. ISBN 5-9221-0157-9
\bibitem{FD}
D.~Fowler, {\it The Binomial Coefficient Function}, The American
Mathematical Monthly, Vol. {\bf 103}, No. 1 (Jan., 1996), pp. 1--17.
DOI.org/10.2307/2975209 Zbl 0857.05003
\bibitem{GRL}
R.L.~Graham, D.E.~Knuth, and O.~Patashnik, {\it Concrete
Mathematics}, Addison-Wesley, 1994. Zbl 0836.00001
\bibitem{Jensen}
J.L.W.V.~Jensen and T.H.~Gronwall, {\it An Elementary Exposition of
the Theory of the Gamma Function}, Annals of Mathematics, Mar.,
1916, Second Series, Vol. {\bf 17}, No. 3 (Mar., 1916), pp.
124--166. DOI.org/10.2307/2007272 Zbl 46.0563.02
\bibitem{Smith}
S.T.~Smith, {\it The binomial coefficient $\binom{n}{x}$ for
arbitrary $x$}, Online Journal of Analytic Combinatorics, December
2020. https://hosted.math.rochester.edu/ojac/vol15/176.pdf Zbl
1468.11069
\end{thebibliography}
\end{document}